\documentclass[twoside,sort&compress]{article}
\usepackage{hyperref}
\usepackage{fancyhdr}
\usepackage{amsmath}
\usepackage{amssymb}
\usepackage{geometry}
\usepackage{cite}
\usepackage{graphicx}

\newtheorem{theorem}{Theorem}[section]

\newtheorem{lemma}[theorem]{Lemma}

\newenvironment{proof}[1][Proof]{\noindent\textbf{#1.} }{\ \rule{0.5em}{0.5em}}
\numberwithin{equation}{section}
\begin{document}

\setcounter{page}{1} \thispagestyle{plain}


\begin{center}
{\Large The Boundedness of Bessel-Riesz Operators\\
on Generalized Morrey Spaces}

\bigskip

{\large M. Idris$^{1}$, H. Gunawan$^{2}$, and Eridani$^3$ }

\bigskip

$^{1,2}$Department of Mathematics, Bandung Institute of Technology\\[0pt]
Bandung 40132, Indonesia

$^{3}$Department of Mathematics, Airlangga University\\[0pt]
Campus C, Mulyorejo, Surabaya 60115, Indonesia

E-mail: $^1$mochidris@students.itb.ac.id, $^2$hgunawan@math.itb.ac.id,
$^3$eri.campanato@gmail.com

\bigskip
\end{center}

\begin{quote}
\textbf{Abstract.} In this paper, we prove the boundedness of Bessel-Riesz
operators on generalized Morrey spaces. The proof uses the usual dyadic
decomposition, a Hedberg-type inequality for the operators, and the boundedness
of Hardy-Littlewood maximal operator. Our results reveal that the norm of
the operators is dominated by the norm of the kernels.


\end{quote}


\section{Introduction}

We begin with the definition of Bessel-Riesz operators. For $0<\gamma $ and
$0<\alpha <n$, we define
$$
I_{\alpha ,\gamma }f\left( x\right)
:=\int_{\mathbb{R}^{n}}K_{\alpha ,\gamma }\left( x-y\right)
f\left( y\right) dy,$$ for every $f\in L_{loc}^{p}\left(
\mathbb{R}^{n}\right) ,\ p\geq 1$,  where $\ K_{\alpha,
,\gamma }\left( x\right) :=\frac{\left\vert x\right\vert ^{\alpha -n}}{
\left( 1+\left\vert x\right\vert \right) ^{\gamma }}$. Here, $I_{\alpha,
\gamma}$ is called \emph{Bessel-Riesz operator} and $K_{\alpha, \gamma}$
is called \emph{Bessel-Riesz kernel}. The name of the kernel resembles
the product of Bessel kernel and Riesz kernel \cite{Stein}. While the Riesz
kernel captures the local behaviour, the Bessel kernel take cares the global
behaviour of the function. The Bessel-Riesz kernel is used in studying
the behaviour of the solution of a Schr\"odinger type equation \cite{Kurata}.

For $\gamma=0$, we have $I_{\alpha ,0}=I_{\alpha }$, known as \textit{%
fractional integral operators }or \textit{Riesz potentials
}\cite{Stein1}. Around 1930, Hardy and Littlewood \cite{Hardy1, Hardy2} and
Sobolev \cite{Sobolev} have proved the boundedness of $I_{\alpha }$ on
Lebesgue spaces via the inequality
$$
\left\Vert I_{\alpha }f\right\Vert _{L^{q}}\leq C\,\left\Vert
f\right\Vert _{L^{p}},
$$
for every $f \in L^{p}\left( \mathbb{R}^{n}\right) $, where $1<p<
\frac{\alpha }{n}$, and $\frac{1}{q}=\frac{1}{p}-\frac{\alpha
}{n}$. Here $C$ denotes a constant which may depend on $\alpha, p,q$, and
$n$, but not on $f$.

For $1\leq p\leq q$, the \textit{(classical) Morrey space} $L^{p,q}\left(
\mathbb{R}^{n}\right)$ is
defined by
$$
L^{p,q}\left( \mathbb{R}^{n}\right) :=\left\{ f \in
L_{loc}^{p}\left( \mathbb{R}^{n}\right)  :\left\Vert f\right\Vert
_{L^{p,q}}<\infty \right\},
$$
where $\| f\|_{L^{p,q}}:=\sup\limits_{r>0, a\in \mathbb{R}^{n}}r^{n\left(
1/q-1/p\right) }\left( \int_{|x-a| < r}|f(x)|^{p}dx\right) ^{1/p}$.
For these spaces, we have an inclusion property which is presented by the
following theorem.

\begin{theorem}
\label{T:Kb copy(a1)} For $1\leq p\leq q$, we have
$L^{q}\left( \mathbb{R}^{n}\right) =L^{q,q}\left(
\mathbb{R}^{n}\right) \subseteq L^{p,q}\left(
\mathbb{R}^{n}\right) \subseteq L^{1,q}\left(
\mathbb{R}^{n}\right)$.
\end{theorem}

\noindent On Morrey spaces\textit{, }Spanne \cite{Peetre} has shown that
$I_{\alpha }$ is bounded form $L^{p_{1},q_{1}}\left(
\mathbb{R}^{n}\right) $ to $L^{p_{2},q_{2}}\left(\mathbb{R}^{n}\right)$
for $1<p_{1}<q_{1}<$ $\frac{n}{\alpha }$, $\frac{1}{p_{2}}=
\frac{1}{p_{1}}-\frac{\alpha }{n}$, and $\frac{1}{q_{2}}=\frac{1}{q_{1}}-
\frac{\alpha }{n}$. Furthermore, Adams \cite{Adams} and
Chiarenza and Frasca \cite{Chiarenza} reproved it and obtained a stronger
result which is presented by the following theorem.

\begin{theorem}
\label{T:Kb copy(1)} [Adams, Chiarenza-Frasca] If $0<\alpha <n$, then we
have
$$
\left\Vert I_{\alpha }f\right\Vert _{L^{p_{2},q_{2}}}\leq
C\,\left\Vert f\right\Vert _{L^{p_{1},q_{1}}},
$$
for every $f\in L^{p_{1},q_{1}}\left(\mathbb{R}^{n}\right)$ where
$1<p_{1}<q_{1}<$ $\frac{n}{\alpha }$,$\frac{1}{p_{2}}=
\frac{1}{p_{1}}\left( 1-\frac{\alpha q_{1}}{n}\right)$, and
$\frac{1}{q_{2}}=\frac{1}{q_{1}}-\frac{\alpha }{n}$.
\end{theorem}

For $\phi :\mathbb{R}^{+}\rightarrow \mathbb{R}^{+}$ and $1\leq p<\infty$,
we define the {\it generalized Morrey space}
$$
L^{p,\phi }\left( \mathbb{R}^{n}\right) :=\left\{ f \in
L_{loc}^{p}\left( \mathbb{R}^{n}\right)  :\left\Vert f\right\Vert
_{L^{p,\phi }}<\infty \right\}, $$
where $\| f\|_{L^{p,\phi }}:=\sup\limits_{r>0,a \in \mathbb{R}^{n}}\frac{1}{\phi(r)
}\left(\frac{1}{r^{n}}\int_{\left\vert x-a\right\vert <r}\left\vert f(x)
\right\vert ^{p}dx\right) ^{1/p}$. Here we assume that $\phi$ is almost
decreasing and $t^{n/p}\phi(t)$ is almost decreasing, so that $\phi$ satisfies
the {\it doubling condition}, that is, there exists a constant $C$ such that
$\frac{1}{C} \leq \frac{\phi(r)}{\phi(v)}\leq C$ whenever
$\frac{1}{2} \leq \frac{r}{v}\leq 2$.

In 1994, Nakai \cite{Nakai94} obtained the boundedness of $I_{\alpha }$ from
$L^{p_{1},\phi }\left(\mathbb{R}^{n}\right)$ to $L^{p_{2},\psi }\left(
\mathbb{R}^{n}\right)$ where  $1<p_{1}<$ $\frac{n}{\alpha }$, $\frac{1}{p_{2}}=
\frac{1}{p_{1}}-\frac{\alpha }{n}$ and $\int_{r}^{\infty
}v^{\alpha -1}\phi(v) dv\leq Cr^{\alpha }\phi(r) \leq C\psi(r)$ for every $r>0$.
Nakai's result may be viewed as an extension of Spanne's.
Later on, in 2009, Gunawan and Eridani \cite{Gunawan} extended Adams-Chiarenza-Frasca's
result.

\begin{theorem}
\label{T:Kbge)} [Gunawan-Eridani] If $\int_{r}^{\infty }\frac{\phi \left( v\right)
}{v}dv \leq C\phi \left( r\right) $, and $\phi(r) \leq Cr^{\beta }$
for every $r>0$, $-\frac{n}{p_{1} }\leq \beta <-\alpha $,
$1<p_{1}<\frac{n}{\alpha}$,  $0<\alpha
<n$, then we have
$$
\| I_{\alpha }f\|_{L^{p_{2},\psi }}\leq C\,\| f\|_{L^{p_{1},\phi }}
$$
for every $f\in L^{p_{1},\phi }\left(\mathbb{R}^{n}\right) $ where
$p_{2}=\frac{\beta p_{1}}{\alpha+\beta }$ and $\psi(r)
=\phi(r)^{p_{1}/p_{2}}$, $r>0$.
\end{theorem}

The proof of the boundedness of $I_{\alpha}$ on Lebesgue spaces, Morrey spaces,
or generalized Morrey spaces, usually involves \emph{Hardy-Littlewood maximal
operator}, which is defined by
$$
Mf\left( x\right) :=\sup_{B \ni x}\frac{1}{\left\vert B\right\vert }%
\int_{B}\left\vert f\left( y\right) \right\vert dy,
$$
for every $f\in L_{loc}^{p}\left(\mathbb{R}^{n}\right) $
where $| B |$ denotes the Lebesgue measure of the ball $B=B(a,r)$
(centered at $a\in \mathbb{R}^{n}$ with radius $r>0$). It is well known
that the operator $M$ is bounded on $L^{p}\left(\mathbb{R}^{n}\right)$
for $1<p\le\infty$ \cite{Stein1, Stein}
and also on Morrey spaces $L^{p,q}$ for $1<p\le q\le\infty$ \cite{Chiarenza}.

Next, we know that $I_{\alpha ,\gamma}$ is guaranteed to be bounded on generalized
Morrey spaces because $K_{\alpha ,\gamma}(x)  \leq K_{\alpha }(x)$ for every
$x \in \mathbb{R}^{n}$. The boundedness of $I_{\alpha,\gamma}$ on Lebesgue
spaces can also be proved by using Young's inequality, as shown in \cite{Idris}.

\begin{theorem}\cite{Idris} \label{cc:01}
For $0<\gamma $ and $0<\alpha <n$, we have $K_{\alpha ,\gamma } \in  L^{t }\left(
\mathbb{R}^{n}\right)$ whenever $ \frac{n}{n+\gamma -\alpha } < t <\frac{n}{n-\alpha }$.
Accordingly, we have
$$
\| I_{\alpha,\gamma }f\|_{L^{q}}\leq \| K_{\alpha ,\gamma}\|_{L^{t}}\| f\|_{L^{p}}
$$
for every $f\in L^{p}\left( \mathbb{R}^{n}\right) $ where $1\leq p<t^{\prime }$,
$\frac{n}{n+\gamma -\alpha }<t<\frac{n}{n-\alpha }$,
$\frac{1}{q}+1=\frac{1}{p}+\frac{1}{t}$.
\end{theorem}

Using the boundedness of Hardy-Littlewood maximal operator, we also know that
$I_{\alpha,\gamma}$ is bounded on Morrey spaces.

\begin{theorem}\cite{Idris} \label{cm:01}
For $0<\gamma $ and $0<\alpha <n$, we have
$$
\|I_{\alpha,\gamma }f\|_{L^{p_2,q_2}}\leq C\,\| K_{\alpha ,\gamma
}\|_{L^{t}}\| f\|_{L^{p_1,q_1}}
$$
for every $f\in L^{p_1,q_1}\left( \mathbb{R}^{n}\right)$ where
$1< p_1<q_1<t^{\prime }$, $\frac{n}{n+\gamma -\alpha }<t<\frac{n}{n-\alpha }$,
$\frac{1}{p_2}+1=\frac{1}{p_1}+\frac{1}{t}$, and $\frac{1}{q_2}+1=\frac{1}{q_1}+\frac{1}{t}$.
\end{theorem}

In 1999, Kurata \textit{et al.} \cite{Kurata} proved the boundedness of
$W\cdot I_{\alpha,\gamma }$ on generalized Morrey spaces where $W$ is a
multiplication operator. A similar result to Kurata's can be found in \cite{Gunawan}.
In the next section, we shall reprove the boundedness of $I_{\alpha ,\gamma }$ on
generalized Morrey spaces using a Hedberg-type inequality and the boundedness of
Hardy-Littlewood maximal operator on these spaces.

\begin{theorem}
\label{nak:a01} (Nakai)
For $1<p\le \infty$ , we have
$$
\| Mf\|_{L^{p,\psi }}\leq C\,\|f\|_{L^{p,\phi }},
$$
for every $f\in L^{p,\phi }\left(\mathbb{R}^{n}\right)$.
\end{theorem}

Our results show that the norm of Bessel-Riesz operators is dominated
by the norm of their kernels on (generalized) Morrey spaces.

\section{Inequalities For $I_{\protect\alpha,\protect\gamma}$ On Generalized Morrey Spaces}

For $0<\alpha<n$ and $\gamma>0$, one may observe that the kernel $K_{\alpha,\gamma}$
belongs to Lebesgue spaces $L^t(\mathbb{R}^n)$ whenever $\frac{n}{n+\gamma -\alpha }
<t<\frac{n}{n-\alpha }$, where
$$
\sum_{k\in\mathbb{Z}}\frac{(2^{k}R)^{(\alpha-n)t+n}}{(1+2^{k}R)
^{\gamma t}} \sim \|K_{\alpha ,\gamma }\|^t_{L^{t}}
$$
(see \cite{Idris}). With this in mind, we obtain the boundedness of $I_{\alpha,\gamma}$ on
generalized Morrey spaces as in the following theorem.

\begin{theorem}
\label{p:a01}
Let $0<\gamma $ and $0<\alpha <n$. If $\phi(r) \leq Cr^{\beta }$ for every $r>0$,
$-\frac{\alpha t^{\prime }}{p_{1}}\leq \beta <-\alpha $, $1<p_{1}<t^{\prime }$, and
$\frac{n}{n+\gamma -\alpha }<t<\frac{n}{n-\alpha }$, then we have
$$
\| I_{\alpha ,\gamma }f\|_{L^{p_{2},\psi }}\leq
C\,\| K_{\alpha ,\gamma }\|_{L^{t}}\|f\|_{L^{p_{1},\phi }},
$$
for every $f\in L^{p_{1},\phi }(\mathbb{R}^{n})$ where $p_{2}=
\frac{\beta p_{1}}{\alpha +\beta }$, and $\psi(r)=\phi(r)^{p_{1}/p_{2}}$.
\end{theorem}

\begin{proof}
Let $0<\gamma $ and $0<\alpha <n$. Suppose that
$\phi(r) \leq Cr^{\beta }$ for every $r>0$, $-\frac{\alpha t^{\prime }}{p_{1}}
\leq \beta <-\alpha $, $1<p_{1}<t^{\prime }$, $\frac{n}{n+\gamma -\alpha }<t
<\frac{n}{n-\alpha }$. Take $f\in L^{p_{1},\phi}(\mathbb{R}^{n})$ and write
$$
I_{\alpha ,\gamma }f\left(x\right) :=I_{1}\left(x\right) +I_{2}\left(
x\right),
$$
for every $x\in \mathbb{R}^{n}$ where $I_{1}\left( x\right) :=\int_{|x-y| <R}
\frac{| x-y|^{\alpha -n}f(y) }{\left( 1+|x-y|\right) ^{\gamma }}dy$ and $I_{2}\left( x\right)
:=\int_{\left\vert x-y\right\vert \geq R}\frac{\left\vert x-y\right\vert
^{\alpha -n}f\left( y\right) }{\left( 1+\left\vert x-y\right\vert \right)
^{\gamma }}dy$, $R>0$.

Using dyadic decomposition, we have the following estimate for $I_{1}$:
\begin{eqnarray*}
\left\vert I_{1}(x) \right\vert &\leq &\sum_{k=-\infty
}^{-1}\int_{2^{k}R\leq \left\vert x-y\right\vert <2^{k+1}R}\frac{\left\vert
x-y\right\vert ^{\alpha -n}\left\vert f\left( y\right) \right\vert }{\left(
1+\left\vert x-y\right\vert \right) ^{\gamma }}dy\\
&\leq&  C_{1}\sum_{k=-\infty }^{-1}\frac{\left( 2^{k}R\right) ^{\alpha -n}}{%
\left( 1+2^{k}R\right) ^{\gamma }}\int_{2^{k}R\leq \left\vert x-y\right\vert
<2^{k+1}R}\left\vert f\left( y\right) \right\vert dy \\
&\leq&C_{2}\,Mf\left( x\right) \sum_{k=-\infty }^{-1}\frac{\left( 2^{k}R\right)
^{\alpha -n+n/t}\left( 2^{k}R\right) ^{n/t^{\prime }}}{\left(
1+2^{k}R\right) ^{\gamma }}.
\end{eqnarray*}
We then use H\"{o}lder's inequality to get
$$
\left\vert I_{1}\left( x\right) \right\vert \leq C_{2}Mf\left( x\right)
\left( \sum_{k=-\infty }^{-1}\frac{\left( 2^{k}R\right) ^{\left( \alpha
-n\right) t+n}}{\left( 1+2^{k}R\right) ^{\gamma t}}\right) ^{1/t}\left(
\sum_{k=-\infty }^{-1}\left( 2^{k}R\right) ^{n}\right) ^{1/t^{\prime }}.
$$
Because we have
$$
\left( \sum_{k=-\infty }^{-1}\frac{\left( 2^{k}R\right) ^{\left( \alpha
-n\right) t+n}}{\left( 1+2^{k}R\right) ^{\gamma t}}\right) ^{1/t}\leq \left(
\sum_{k\in\mathbb{Z}}\frac{\left( 2^{k}R\right) ^{\left( \alpha -n\right) t+n}}{\left(
1+2^{k}R\right) ^{\gamma t}}\right) ^{1/t}\sim \left\Vert K_{\alpha
,\gamma }\right\Vert _{L^{t}},
$$
we obtain $\left\vert I_{1}(x) \right\vert \leq C_{3}\left\Vert
K_{\alpha ,\gamma }\right\Vert _{L^{t}}Mf\left( x\right) R^{n/t^{\prime }}$.

To estimate $I_{2}$, we use H\"{o}lder's inequality again:
\begin{eqnarray*}
\left\vert I_{2}\left( x\right) \right\vert &\leq&  C_{4}\sum_{k=0}^{\infty }%
\frac{\left( 2^{k}R\right) ^{\alpha -n}}{\left( 1+2^{k}R\right) ^{\gamma }}%
\int_{2^{k}R\leq \left\vert x-y\right\vert <2^{k+1}R}\left\vert f\left(
y\right) \right\vert dy \\
&\leq& C_{4}\sum_{k=0}^{\infty }\frac{\left( 2^{k}R\right) ^{\alpha -n}}{%
\left( 1+2^{k}R\right) ^{\gamma }}\left( 2^{k}R\right) ^{n/p_{1}^{\prime }}
 \left( \int_{2^{k}R\leq \left\vert x-y\right\vert
<2^{k+1}R}\left\vert f\left( y\right) \right\vert ^{p_{1}}dy\right)
^{1/p_{1}}.
\end{eqnarray*}
It follows that
\begin{eqnarray*}
 \left\vert I_{2}\left( x\right) \right\vert
 &\leq& C_{5}\left\Vert f\right\Vert _{L^{p_{1},\phi }}\sum_{k=0}^{\infty }%
\frac{\left( 2^{k}R\right) ^{\alpha -n+n/t}}{\left( 1+2^{k}R\right) ^{\gamma
}}\phi \left( 2^{k}R\right) \left( 2^{k}R\right) ^{n/t^{\prime }}\\  &\leq&
C_{6}\left\Vert f\right\Vert _{L^{p_{1},\phi }}\sum_{k=0}^{\infty }%
\frac{\left( 2^{k}R\right) ^{\alpha -n+n/t}}{\left( 1+2^{k}R\right) ^{\gamma
}}\left( 2^{k}R\right) ^{\beta +n/t^{\prime }}.
\end{eqnarray*}
Another use of H\"{o}lder's inequality gives
$$
\left\vert I_{2}\left( x\right) \right\vert \leq C_{6}\left\Vert
f\right\Vert _{L^{p_{1},\phi }}\left( \sum_{k=0}^{\infty }\frac{\left(
2^{k}R\right) ^{\left( \alpha -n\right) t+n}}{\left( 1+2^{k}R\right)
^{\gamma t}}\right) ^{1/t}\left( \sum_{k=0}^{\infty }\left( 2^{k}R\right)
^{\beta t^{\prime }+n}\right) ^{1/t^{\prime }}.
$$
Because $\beta t^{\prime }+n<0$ and $\sum_{k=0}^{\infty }\frac{\left(
2^{k}R\right) ^{\left( \alpha -n\right) t+n}}{\left( 1+2^{k}R\right)
^{\gamma t}} \lesssim \| K_{\alpha ,\gamma }\|^t_{L^{t}}$, we obtain
$$
\left\vert I_{2}\left( x\right) \right\vert \leq C_{7}\left\Vert K_{\alpha
,\gamma }\right\Vert _{L^{t}}\left\Vert f\right\Vert _{L^{p_{1},\phi
}}R^{\beta }R^{n/t^{\prime }}.
$$

Summing the two estimates, we obtain
$$
| I_{\alpha ,\gamma }f(x) | \leq
C_{8}\left\Vert K_{\alpha ,\gamma }\right\Vert _{L^{t}}\left( Mf\left(
x\right) R^{n/t^{\prime }}+\left\Vert f\right\Vert _{L^{p_{1},\phi
}}R^{n/t^{\prime }+\beta }\right),
$$
for every $x\in \mathbb{R}^{n}$. Now, for each $x\in\mathbb{R}^n$,
choose $R>0$ such that $R^{-\beta }=\frac{\| f\|_{L^{p_{1},\phi }}}{Mf\left(
x\right) }$. Hence we get a Hedberg-type inequality for $I_{\alpha,\gamma}f$,
namely
$$
\left\vert I_{\alpha ,\gamma }f\left( x\right) \right\vert \leq
C_{9}\left\Vert K_{\alpha ,\gamma }\right\Vert _{L^{t}}\left\Vert
f\right\Vert _{L^{p_{1},\phi }}^{-\alpha /\beta }Mf\left( x\right)
^{1+\alpha /\beta }.
$$
Now put $p_{2}:=\frac{\beta p_{1}}{\alpha +\beta }$. For arbitrary $a\in
\mathbb{R}^n$ and $r>0$, we have
$$
\left( \int_{|x-a|<r}\left\vert I_{\alpha ,\gamma
}f(x) \right\vert^{p_{2}}dx\right)^{1/p_{2}}\leq
C_{9}\left\Vert K_{\alpha ,\gamma }\right\Vert _{L^{t}}\left\Vert
f\right\Vert _{L^{p_{1},\phi }}^{1-p_{1}/p_{2}}\left( \int_{\left\vert
x\right\vert <r}\left\vert Mf(x) \right\vert ^{p_{1}}dx\right)
^{1/p_{2} }.
$$
We divide both sides by $\phi(r)^{p_{1}/p_{2}}r^{n/p_{2}}$ to get
$$
\frac{\left( \int_{\left\vert x\right\vert <r}\left\vert I_{\alpha ,\gamma
}f\left( x\right) \right\vert ^{p_{2}}dx\right) ^{1/p_{2}}}{\psi(r)r^{n/p_{2}}}
\leq C_{9}\left\Vert K_{\alpha ,\gamma
}\right\Vert _{L^{t}}\left\Vert f\right\Vert _{L^{p_{1},\phi
}}^{1-p_{1}/p_{2}}   \frac{\left( \int_{\left\vert x\right\vert <r}
|Mf(x)|^{p_{1}}dx\right)^{( 1/p_{2}) }}{\phi(r)^{p_{1}/p_{2}}r^{n/p_{2}}},
$$
where $\psi(r) :=\phi(r)^{p_{1}/p_{2}}$.
Taking the supremum over $a\in\mathbb{R}^n$ and $r>0$, we obtain
$$
\left\Vert I_{\alpha ,\gamma }f\right\Vert _{L^{p_{2},\psi }}\leq
C_{10}\left\Vert K_{\alpha ,\gamma }\right\Vert _{L^{t}}\left\Vert
f\right\Vert _{L^{p_{1},\phi }}^{1-p_{1}/p_{2}}\left\Vert Mf\right\Vert
_{L^{p_{1},\phi }}^{p_{1}/p_{2}}.
$$
By the boundedness of the maximal operator on generalized Morrey spaces (Nakai's
Theorem), the desired result follows:
$\| I_{\alpha ,\gamma }f\|_{L^{p_{2},\psi }}\leq C\,\| K_{\alpha ,\gamma}
\|_{L^{t}}\| f\|_{L^{p_{1},\phi }}.$
\end{proof}

\bigskip

We note that from the inclusion property of Morrey spaces, we have
$\left\Vert I_{\alpha ,\gamma }f\right\Vert _{L^{s,t }} \leq
\left\Vert I_{\alpha ,\gamma }f\right\Vert _{L^{t,t }}=
\left\Vert I_{\alpha ,\gamma }f\right\Vert _{L^{t }}$ where
$1\leq s\leq t$, $\frac{n}{n+\gamma -\alpha }<t<\frac{n}{n-\alpha }$.
Now we wish to find a better upper bound for the norm of
$I_{\alpha,\gamma}$. For this purpose, we shall use the fact that the kernel
$K_{\alpha,\gamma}$ belongs to Morrey spaces.

\begin{theorem}
\label{p:a02}
Let $0<\gamma $ and $0<\alpha <n$. If $\phi(r)\leq Cr^{\beta }$ for every
$r>0$, $-\frac{\alpha t^{\prime }}{p_{1}}\leq \beta <-\alpha $,
$1<p_{1}<t^{\prime }$, $\frac{n}{n+\gamma-\alpha }<t<\frac{n}{n-\alpha }$,
then we have
$$
\left\Vert I_{\alpha ,\gamma }f\right\Vert _{L^{p_{2},\psi }}\leq
C\,\left\Vert K_{\alpha ,\gamma }\right\Vert
_{L^{s,t}}\left\Vert f\right\Vert _{L^{p_{1},\phi }},
$$
for every $f\in L^{p_{1},\phi }\left( \mathbb{R}^{n}\right) $ where
$1\leq s\leq t$, $p_{2}=\frac{\beta p_{1}}{\alpha +\beta }$, and $\psi
\left( v\right) =\phi \left( v\right) ^{p_{1}/p_{2}}$.
\end{theorem}

\begin{proof}
Let $0<\gamma $ and $0<\alpha <n$.
Suppose that $\phi(r) \leq Cr^{\beta }$ for every $r>0$, $-\frac{\alpha
t^{\prime }}{p_{1}}\leq \beta <-\alpha $, $1<p_{1}<t^{\prime }$, $
\frac{n}{n+\gamma -\alpha }<t<\frac{n}{n-\alpha }$.
As in the proof of Theorem \ref{p:a01}, we have $I_{\alpha ,\gamma }f(x) :=
I_{1}\left( x\right) +I_{2}(x)$ for every $x \in  \mathbb{R}^{n}$.
Now, we estimate $I_{1}$ using dyadic decomposition as follow:
\begin{eqnarray*}
\left\vert I_{1}\left( x\right) \right\vert &\leq &\sum_{k=-\infty
}^{-1}\int_{2^{k}R\leq \left\vert x-y\right\vert <2^{k+1}R}\frac{\left\vert
x-y\right\vert ^{\alpha -n}\left\vert f\left( y\right) \right\vert }{\left(
1+\left\vert x-y\right\vert \right) ^{\gamma }}dy \\
&\leq& C_{1}\sum_{k=-\infty }^{-1}\frac{\left( 2^{k}R\right) ^{\alpha -n}}{%
\left( 1+2^{k}R\right) ^{\gamma }}\int_{2^{k}R\leq \left\vert x-y\right\vert
<2^{k+1}R}\left\vert f\left( y\right) \right\vert dy \\
&=&C_{2}Mf\left( x\right) \sum_{k=-\infty }^{-1}\frac{\left( 2^{k}R\right)
^{\alpha -n+n/s}\left( 2^{k}R\right) ^{n/s^{\prime }}}{\left(
1+2^{k}R\right) ^{\gamma }},
\end{eqnarray*}
where $1\leq s\leq t$. By H\"{o}lder's inequality,
$$
\left\vert I_{1}\left( x\right) \right\vert \leq C_{2}Mf\left( x\right)
\left( \sum_{k=-\infty }^{-1}\frac{\left( 2^{k}R\right) ^{\left( \alpha
-n\right) s+n}}{\left( 1+2^{k}R\right) ^{\gamma s}}\right) ^{1/s}\left(
\sum_{k=-\infty }^{-1}\left( 2^{k}R\right) ^{n}\right) ^{1/s^{\prime }}.
$$
We also have $\sum_{k=-\infty }^{-1}\frac{\left( 2^{k}R\right)
^{\left( \alpha -n\right) s+n}}{\left( 1+2^{k}R\right) ^{\gamma s}}
\lesssim \int_{0<\left\vert x\right\vert <R}K_{\alpha ,\gamma
}^{s}\left( x\right) dx,$ so that
$$
\left\vert I_{1}\left( x\right) \right\vert \leq C_{3}Mf\left( x\right)
\left( \int_{0<\left\vert x\right\vert <R}K_{\alpha ,\gamma }^{s}\left(
x\right) dx\right) ^{\frac{1}{s}}R^{n/s^{\prime }}
 \leq C_{3}\left\Vert K_{\alpha ,\gamma }\right\Vert _{L^{s,t}}Mf\left(
x\right) R^{n/t^{\prime }}.
$$

Next, we estimate $I_{2}$ by using H\"{o}lder's inequality. As in the
proof of Theorem \ref{p:a01}, we obtain
\begin{eqnarray*}
\left\vert I_{2}\left( x\right) \right\vert &\leq &  C_{4}\sum_{k=0}^{\infty }
\frac{\left( 2^{k}R\right) ^{\alpha -n}}{\left( 1+2^{k}R\right) ^{\gamma }}
\left( 2^{k}R\right) ^{n/p_{1}^{\prime }}
 \left( \int_{2^{k}R\leq \left\vert x-y\right\vert
<2^{k+1}R}\left\vert f\left( y\right) \right\vert ^{p_{1}}dy\right)^{1/p_{1}}.
\end{eqnarray*}
It thus follows that
\begin{eqnarray*}
\left\vert I_{2}\left( x\right) \right\vert &\leq &C_{5}\|f\|_{L^{p_{1},\phi }}
\sum_{k=0}^{\infty }\frac{\left( 2^{k}R\right)^{\alpha }\phi(2^{k}R)}{\left(
1+2^{k}R\right) ^{\gamma }}\frac{\left( \int_{2^{k}R\leq
|x-y| <2^{k+1}R}dy\right) ^{1/s}}{\left( 2^{k}R\right)^{n/s}} \\
&\leq &C_{6}\|f\|_{L^{p_{1},\phi }}\sum_{k=0}^{\infty
}\phi(2^{k}R) \left( 2^{k}R\right)^{n/t^{\prime }}\frac{\left(
\int_{2^{k}R\leq |x-y| <2^{k+1}R}K_{\alpha,\gamma
}^{s}(x-y) dy\right)^{1/s}}{\left( 2^{k}R\right)^{n/s-n/t}}
\end{eqnarray*}%
Because $\phi(r) \leq Cr^{\beta }$ and $\frac{\left(
\int_{2^{k}R\leq |x-y| <2^{k+1}R}K_{\alpha,\gamma}^{s}(x-y) dy\right)^{1/s}}
{\left( 2^{k}R\right) ^{n/s-n/t}}
\lesssim \left\Vert K_{\alpha ,\gamma }\right\Vert _{L^{s,t}}$ for every
$k=0,1,2,\dots$, we get
\begin{eqnarray*}
\left\vert I_{2}\left( x\right) \right\vert &\leq& C_{7}\left\Vert K_{\alpha
,\gamma }\right\Vert _{L^{s,t}}\left\Vert f\right\Vert _{L^{p_{1},\phi
}}\sum_{k=0}^{\infty }\left( 2^{k}R\right) ^{\beta +n/t^{\prime }}  \\
&\leq& C_{8}\left\Vert K_{\alpha ,\gamma }\right\Vert _{L^{s,t}}\left\Vert
f\right\Vert _{L^{p_{1},\phi }}R^{\beta }R^{n/t^{\prime }}.
\end{eqnarray*}

From the two estimates, we obtain
$$
| I_{\alpha ,\gamma }f(x) | \leq
C_{9}\left\Vert K_{\alpha ,\gamma }\right\Vert _{L^{s,t}}\left( Mf(x)
R^{n/t^{\prime }}+\|f\|_{L^{p_{1},\phi}}R^{n/t^{\prime }+\beta}\right),
$$
for every $x\in \mathbb{R}^{n}$. Now, for each $x\in \mathbb{R}^n$, choose $R>0$
such that $R^{-\beta }=\frac{\|f\|_{L^{p_{1},\phi }}}{Mf(x)}$. Hence we get
$$
\left\vert I_{\alpha ,\gamma }f\left( x\right) \right\vert \leq
C_{9}\left\Vert K_{\alpha ,\gamma }\right\Vert _{L^{s,t}}\left\Vert
f\right\Vert _{L^{p_{1},\phi }}^{-\alpha /\beta }Mf\left( x\right)
^{1+\alpha /\beta }.
$$
Put $p_{2}:=\frac{\beta p_{1}}{\alpha +\beta }$. For arbitrary $a\in
\mathbb{R}^n$ and $r>0$, we have
$$
\left( \int_{\left\vert x-a\right\vert <r}\left\vert I_{\alpha ,\gamma
}f\left( x\right) \right\vert ^{p_{2}}dx\right) ^{1/p_{2}}\leq
C_{9}\left\Vert K_{\alpha ,\gamma }\right\Vert _{L^{s,t}}\left\Vert
f\right\Vert _{L^{p_{1},\phi }}^{1-p_{1}/p_{2}}\left( \int_{\left\vert
x-a\right\vert <r}\left\vert Mf\left( x\right) \right\vert ^{p_{1}}dx\right)
^{1/p_{2} }.
$$
Divide both sides by
$\phi \left( r\right) ^{p_{1}/p_{2}}r^{n/p_{2}}$ and take the supremum
over $a\in\mathbb{R}^n$ and $r>0$ to get
$$
\left\Vert I_{\alpha ,\gamma }f\right\Vert _{L^{p_{2},\psi }}\leq
C_{10}\left\Vert K_{\alpha ,\gamma }\right\Vert _{L^{s,t}}\left\Vert
f\right\Vert _{L^{p_{1},\phi }}^{1-p_{1}/p_{2}}\left\Vert Mf\right\Vert
_{L^{p_{1},\phi }}^{p_{1}/p_{2}},
$$
where $\psi(r) :=\phi(r)^{p_{1}/p_{2}}$. With the boundedness of the maximal
operator on generalized Morrey spaces (Nakai's
Theorem), we obtain the desired result: $\left\Vert I_{\alpha ,\gamma }f\right\Vert
_{L^{p_{2},\psi }}\leq C_{p_{1},\phi }\left\Vert K_{\alpha ,\gamma
}\right\Vert _{L^{s,t}}\left\Vert f\right\Vert _{L^{p_{1},\phi }}$.
\end{proof}

\bigskip

Note that by Theorem \ref{p:a02} and the inclusion of Morrey spaces, we
recover Theorem \ref{p:a01}:
$$
\left\Vert I_{\alpha ,\gamma }f\right\Vert
_{L^{p_{2},\psi }}\leq C\,\left\Vert K_{\alpha ,\gamma
}\right\Vert _{L^{s,t}}\left\Vert f\right\Vert _{L^{p_{1},\phi }}\leq
C\,\left\Vert K_{\alpha ,\gamma
}\right\Vert _{L^{t}}\left\Vert f\right\Vert _{L^{p_{1},\phi }}.
$$

We wish to obtain a better estimate.  The following lemma presents that
the Bessel-Riesz kernels belong to generalized Morrey space $L^{s,\sigma}
\left( \mathbb{R}^{n}\right)$ for some $s \geq 1$ and a suitable function $\sigma$.

\begin{lemma}\label{lemgbr:02}
Suppose that  $0<\gamma$ and $0<\alpha<n$. If $ \sigma: \mathbb{R}^{+}\rightarrow \mathbb{R}^{+}$
satisfies $\int_{0<r \leq R} r^{(\alpha-n)s+n-1}dr \leq C \sigma^s(R) R^n$
for every $R>0$, then  $K_{\alpha,\gamma} \in L^{s,\sigma }\left( \mathbb{R}^{n}\right)$.
\end{lemma}

\begin{proof}
Suppose that the hypothesis holds. We observe that
\begin{eqnarray*}
\int_{|x-0| \leq R} K_{\alpha,\gamma}^s (x)dx &=& \int_{|x|\leq R} \frac{|x|^{(\alpha-n) s}}{ (1+|x|)^{\gamma s}}dx
\leq  C  \int_{0<r \leq R}  r^{(\alpha-n)s+n-1}dr  \leq C \sigma^s(R) R^n.
\end{eqnarray*}
We divide both sides of the inequality by $\sigma^s(R)R^n$ and take $s^{\rm th}$-root to obtain
$\frac{\left(\int_{|x-0| \leq R} K_{\alpha,\gamma}^s (x)dx  \right)^{1/s}}{\sigma(R)R^{n/s}} \leq C^{1/s}.$
Now, taking the supremum over $R>0$, we have
$\sup\limits_{R>0}\frac{\left(\int_{|x-0| \leq R} K_{\alpha,\gamma}^s (x)dx
\right)^{1/s}}{\sigma(R)R^{n/s}} <\infty.$ Hence $K_{\alpha, \gamma} \in L^{s,\sigma }\left( \mathbb{R}^{n}\right) $.
\end{proof}

By the hypothesis of Lemma \ref{lemgbr:02} we also obtain
$  \frac{\left( \int_{2^{k}R < |x| \leq 2^{k+1}R } K_{\rho,\gamma}^s (x) dx
\right)^{1/s}}{\sigma ( 2^{k}R) \left(  2^{k}R\right)^{n/s}} \lesssim
\left\Vert K_{\alpha ,\gamma }\right\Vert _{L^{s,\sigma}}$ for every integer $k$ and $R>0$.
Moreover, $\frac{ \left(\sum_{k=-1 }^{-\infty } K_{\alpha ,\gamma }^{s}
\left(  2^{k}R\right) \left(  2^{k}R\right)^n \right)^{1/s}}{\sigma (R)
\left(  R\right)^{n/s}} \lesssim \left\Vert K_{\alpha ,\gamma }\right\Vert_{L^{s,\sigma}}$ holds
for every $R>0$.
One may observe that $1 \leq s \leq \frac{n \ln R_1 }{- \ln \sigma (R_1) }$ for every $R_1>1$.
For $\sigma(R)=R ^{-n/t}$, this inequality reduces to $1 \leq s \leq t$.

We shall now use the lemma to prove the following theorem.


\begin{theorem}
\label{p:a03} Suppose that $\sigma :\mathbb{R}^{+}\rightarrow \mathbb{R}^{+}$
satisfies the doubling condition and $\sigma(r) \leq Cr^{-\alpha }$ for every $r>0$,
so that $K_{\alpha ,\gamma }\in L^{s,\sigma }\left(
\mathbb{R}^{n}\right)$ for $1\leq s<\frac{n}{n-\alpha }$, where $0<\alpha <n$ and
$\gamma >0$. If $\ \phi(r) \leq
Cr^{\beta }$ for every $r>0$, where  $-\frac{n}{p_{1}}<\beta <-\alpha$,
then we have
\[
\left\Vert I_{\alpha ,\gamma }f\right\Vert _{L^{p_{2},\psi }}\leq
C_{p_{1},\phi }\left\Vert K_{\alpha ,\gamma }\right\Vert _{L^{s,\sigma
}}\left\Vert f\right\Vert _{L^{p_{1},\phi }}
\]
for every $f\in L^{p_{1},\phi }\left( \mathbb{R}^{n}\right) $, where $1<p_{1}<
\frac{n}{\alpha }$, $p_{2}=\frac{\beta p_{1}}{\beta
+n-\alpha }$ and $\psi(r) =\phi(r) ^{p_{1}/p_{2}}$.
\end{theorem}

\begin{proof}
Let $0<\gamma $ and $0<\alpha <n$. Suppose that $\sigma :\mathbb{R}^{+}\rightarrow
\mathbb{R}^{+}$ satisfies the doubling condition and $\sigma(r) \leq
Cr^{-\alpha}$ for every $r>0$, such that $K_{\alpha ,\gamma }\in L^{s,\sigma }\left(
\mathbb{R}^{n}\right) $ for $1\leq s<\frac{n}{n-\alpha }$. Suppose also that
$\phi(r) \leq Cr^{\beta }$ for every $r>0$, where $-\frac{n}{p_{1}}<\beta
<-\alpha $, $1<p_{1}<\frac{n}{\alpha }$. As in the proof of Theorem \ref{p:a01},
we write $I_{\alpha ,\gamma }f\left( x\right) :=I_{1}\left( x\right) +I_{2}\left(
x\right)$ for every $x \in  \mathbb{R}^{n}$.
As usual, we estimate $I_{1}$ by using dyadic decomposition:
\begin{eqnarray*}
\left\vert I_{1}\left( x\right) \right\vert  &\leq &\sum_{k=-\infty
}^{-1}\int_{2^{k}R\leq \left\vert x-y\right\vert <2^{k+1}R}\frac{\left\vert
x-y\right\vert ^{\alpha -n}\left\vert f\left( y\right) \right\vert }{\left(
1+\left\vert x-y\right\vert \right) ^{\gamma }}dy \\
&\leq &C_{1}\sum_{k=-\infty }^{-1}\frac{\left( 2^{k}R\right) ^{\alpha -n}}{%
\left( 1+2^{k}R\right) ^{\gamma }}\int_{2^{k}R\leq \left\vert x-y\right\vert
<2^{k+1}R}\left\vert f\left( y\right) \right\vert dy \\
&=&C_{2}Mf\left( x\right) \sum_{k=-\infty }^{-1}\frac{\left( 2^{k}R\right)
^{\alpha -n+n/s}\left( 2^{k}R\right) ^{n/s^{\prime }}}{\left(
1+2^{k}R\right) ^{\gamma }}
\end{eqnarray*}
By using H\"{o}lder inequality, we obtain
$$
\left\vert I_{1}\left( x\right) \right\vert \leq C_{2}Mf\left( x\right)
\left( \sum_{k=-\infty }^{-1}\frac{\left( 2^{k}R\right) ^{\left( \alpha
-n\right) s+n}}{\left( 1+2^{k}R\right) ^{\gamma s}}\right) ^{1/s}\left(
\sum_{k=-\infty }^{-1}\left( 2^{k}R\right) ^{n}\right) ^{1/s^{\prime }}.
$$
But $\sum_{k=-\infty }^{-1}\frac{\left( 2^{k}R\right)
^{\left( \alpha -n\right) s+n}}{\left( 1+2^{k}R\right) ^{\gamma s}}\lesssim
\int_{0<\left\vert x\right\vert <R}K_{\alpha ,\gamma
}^{s}\left( x\right) dx$, and so we get
\begin{eqnarray*}
\left\vert I_{1}\left( x\right) \right\vert  &\leq &C_{2}Mf\left( x\right)
\left( \int_{0<\left\vert x\right\vert <R}K_{\alpha ,\gamma }^{s}\left(
x\right) dx\right) ^{\frac{1}{s}}R^{n/s^{\prime }} \\
&\leq &C_{2}\left\Vert K_{\alpha ,\gamma }\right\Vert _{L^{s,\sigma
}}Mf\left( x\right) \sigma \left( R\right) R^{n} \\
&\leq &C_{3}\left\Vert K_{\alpha ,\gamma }\right\Vert _{L^{s,\sigma
}}Mf\left( x\right) R^{n-\alpha }.
\end{eqnarray*}

Next, we estimate $I_{2}$ as follows:
\begin{eqnarray*}
\left\vert I_{2}\left( x\right) \right\vert &\leq &C_{4}\sum_{k=0}^{\infty }%
\frac{\left( 2^{k}R\right) ^{\alpha -n}}{\left( 1+2^{k}R\right) ^{\gamma }}%
\int_{2^{k}R\leq \left\vert x-y\right\vert <2^{k+1}R}\left\vert f\left(
y\right) \right\vert dy \\
&\leq &C_{4}\sum_{k=0}^{\infty }\frac{\left( 2^{k}R\right) ^{\alpha -n}}{%
\left( 1+2^{k}R\right) ^{\gamma }}\left( 2^{k}R\right) ^{n/p_{1}^{\prime }}
 \left( \int_{2^{k}R\leq \left\vert x-y\right\vert
<2^{k+1}R}\left\vert f\left( y\right) \right\vert ^{p_{1}}dy\right)
^{1/p_{1}} \\
&\leq &C_{5}\left\Vert f\right\Vert _{L^{p_{1},\phi }}\sum_{k=0}^{\infty }%
\frac{\left( 2^{k}R\right) ^{\alpha }\phi \left( 2^{k}R\right) \left( 2^{k}R\right) ^n }{\left(
1+2^{k}R\right) ^{\gamma }}\frac{\left( \int_{2^{k}R\leq \left\vert
x-y\right\vert <2^{k+1}R}dy\right) ^{1/s}}{\left( 2^{k}R\right) ^{n/s}} \\
&\leq &C_{6}\left\Vert f\right\Vert _{L^{p_{1},\phi }}\sum_{k=0}^{\infty
}\left( 2^{k}R\right) ^{n-\alpha +\beta }\frac{\left( \int_{2^{k}R\leq
\left\vert x-y\right\vert <2^{k+1}R}K_{\alpha ,\gamma }^{s}\left( x-y\right)
dy\right) ^{1/s}}{\sigma \left( 2^{k}R\right) \left( 2^{k}R\right) ^{n/s}}
\end{eqnarray*}%
Because $\frac{\int_{2^{k}R\leq \left\vert x-y\right\vert
<2^{k+1}R}K_{\alpha ,\gamma }^{s}\left( x-y\right) dy}{\sigma
\left( 2^{k}R\right) \left( 2^{k}R\right) ^{n}}\lesssim \|
K_{\alpha ,\gamma }\|^s_{L^{s,\sigma }}$ for every $k=0,1,2,\dots$, we obtain
\begin{eqnarray*}
\left\vert I_{2}\left( x\right) \right\vert &\leq &C_{6}\left\Vert K_{\alpha
,\gamma }\right\Vert _{L^{s,\sigma }}\left\Vert f\right\Vert _{L^{p_{1},\phi
}}\sum_{k=0}^{\infty }\left( 2^{k}R\right) ^{n-\alpha +\beta } \\
&\leq &C_{7}\left\Vert K_{\alpha ,\gamma }\right\Vert _{L^{s,\sigma
}}\left\Vert f\right\Vert _{L^{p_{1},\phi }}R^{n-\alpha +\beta }.
\end{eqnarray*}

It follows from the above estimates for $I_1$ and $I_2$ that
$$
\left\vert I_{\alpha ,\gamma }f\left( x\right) \right\vert \leq
C_{8}\left\Vert K_{\alpha ,\gamma }\right\Vert _{L^{s,\sigma }}\left(
Mf\left( x\right) R^{n-\alpha }+\left\Vert f\right\Vert _{L^{p_{1},\phi
}}R^{n-\alpha +\beta }\right)
$$
for every $x\in \mathbb{R}^{n}$. Now, for each $x\in\mathbb{R}^n$, choose
$R>0$ such that $R^{-\beta }=\frac{\| f\|_{L^{p_{1},\phi }}}{Mf\left(
x\right) }$, whence
$$
\left\vert I_{\alpha,\gamma }f\left( x\right) \right\vert \leq
C_{9}\left\Vert K_{\alpha ,\gamma }\right\Vert _{L^{s,\sigma }}\left\Vert
f\right\Vert _{L^{p_{1},\phi }}^{\left( \alpha -n\right) /\beta }Mf\left(
x\right) ^{1+\left( n-\alpha \right) /\beta }.
$$
Put $p_{2}:=\frac{\beta p_{1}}{\beta +n-\alpha }$. For arbitrary $a\in
\mathbb{R}^n$ and $r>0$, we have
\[
\left( \int_{\left\vert x-a\right\vert <r}\left\vert I_{\alpha ,\gamma
}f\left( x\right) \right\vert ^{p_{2}}dx\right) ^{1/p_{2}}\leq
C_{9}\left\Vert K_{\alpha ,\gamma }\right\Vert _{L^{s,\sigma }}\left\Vert
f\right\Vert _{L^{p_{1},\phi }}^{1-p_{1}/p_{2}}\left( \int_{\left\vert
x-a\right\vert <r}\left\vert Mf\left( x\right) \right\vert ^{p_{1}}dx\right)
^{\left( 1/p_{2}\right) }.
\]
Divide the both sides by $\phi \left( r\right) ^{p_{1}/p_{2}}r^{n/p_{2}}$ to get
\begin{eqnarray*}
\frac{\left( \int_{\left\vert x\right\vert <r}\left\vert I_{\alpha ,\gamma
}f\left( x\right) \right\vert ^{p_{2}}dx\right) ^{1/p_{2}}}{\psi \left(
r\right) r^{n/p_{2}}} &=&\frac{\left( \int_{\left\vert x\right\vert <r}\left\vert I_{\alpha ,\gamma
}f\left( x\right) \right\vert ^{p_{2}}dx\right) ^{1/p_{2}}}{\phi \left(
r\right) ^{p_{1}/p_{2}}r^{n/p_{2}}} \\
&\leq& C_{9}\left\Vert K_{\alpha ,\gamma
}\right\Vert _{L^{s,\sigma }}\left\Vert f\right\Vert _{L^{p_{1},\phi
}}^{1-p_{1}/p_{2}} \frac{\left( \int_{\left\vert x\right\vert <r}\left\vert Mf\left(
x\right) \right\vert ^{p_{1}}dx\right) ^{\left( 1/p_{2}\right) }}{\phi
\left( r\right) ^{p_{1}/p_{2}}r^{n/p_{2}}},
\end{eqnarray*}
where $\psi(r) :=\phi(r)^{p_{1}/p_{2}}$. Finally, take the supremum
over $a\in\mathbb{R}^n$ and $r>0$ to obtain
$$
\left\Vert I_{\alpha,\gamma }f\right\Vert _{L^{p_{2},\psi }}\leq
C_{10}\left\Vert K_{\alpha,\gamma }\right\Vert _{L^{s,\sigma }}\left\Vert
f\right\Vert _{L^{p_{1},\phi }}^{1-p_{1}/p_{2}}\left\Vert Mf\right\Vert
_{L^{p_{1},\phi }}^{p_{1}/p_{2}}.
$$
Because the maximal operator is bounded on generalized Morrey spaces (Nakai's
Theorem), we conclude that $\left\Vert I_{\alpha ,\gamma }f\right\Vert
_{L^{p_{2},\psi }}\leq C_{p_{1},\phi }\left\Vert K_{\alpha ,\gamma
}\right\Vert _{L^{s,\sigma }}\left\Vert f\right\Vert _{L^{p_{1},\phi }}$.
\end{proof}

\section{Concluding Remarks}

The results presented in this paper, namely Theorems \ref{p:a01}, \ref{p:a02},
and \ref{p:a03}, extend the results on the boundedness of Bessel-Riesz
operators on Morrey spaces \cite{Idris}. Similar to Gunawan-Eridani's result for $I_{\alpha}$,
Theorems \ref{p:a01}, \ref{p:a02}, and \ref{p:a03} ensures that
$I_{\alpha,\gamma}: L^{p_{1},\phi }\left(\mathbb{R}^{n}\right) \rightarrow
L^{p_{2},\phi^{p_1/p_2} }\left(\mathbb{R}^{n}\right)$. Notice that if we have
$\sigma:\mathbb{R}^+ \to \mathbb{R}^+$ such that for $t \in (\frac{n}{n+\gamma-\alpha},
\frac{n}{n-\alpha})$, then $ R^{-n/t}<\sigma(R)$ holds for every $R>0$, then Theorem
\ref{p:a03} gives a better estimate than Theorem \ref{p:a02}. Now, if we define
$\sigma(R) :=\left(1+R^{n/t_1}\right)R^{-n/t}$ for some $t_1>t$,
then $\left\Vert K_{\alpha ,\gamma}\right\Vert _{L^{s,\sigma }}<\left\Vert K_{\alpha ,\gamma
}\right\Vert _{L^{s,t }}$. By Theorem \ref{p:a02} and the inclusion property of
Morrey spaces, we obtain
\begin{eqnarray*}
\| I_{\alpha ,\gamma }f\|_{L^{p_{2},\psi }}  &\leq&
C\,\| K_{\alpha ,\gamma}\|_{L^{s,\sigma }}\| f\|_{L^{p_{1},\phi }}\\
&<&  C\,\| K_{\alpha ,\gamma}\|_{L^{s,t }}\| f\|_{L^{p_{1},\phi }}\\
&\leq&  C\,\| K_{\alpha ,\gamma}\|_{L^{t }}\| f\|_{L^{p_{1},\phi }}.
\end{eqnarray*}
We can therefore say that Theorem \ref{p:a03} gives the best estimate among the three.
Furthermore, we have also shown that, in each theorem, the norm of Bessel-Riesz
operators on generalized Morrey spaces is dominated by that of Bessel-Riesz kernels.


\bigskip

\noindent{\bf Acknowledgements}. The first and second authors are supported by
ITB Research \& Innovation Program 2015.

\end{document}